\documentclass[12pt]{amsart}
\usepackage{amssymb}
\usepackage{amscd}
\usepackage{eucal}
\usepackage{mathrsfs}
\usepackage{a4wide}
\usepackage[all]{xy} 

\newcommand\wLt{\ensuremath{\textrm{wL-tight}}}
\newcommand\ccct{\ensuremath{\textrm{ccc-tight}}}
\newcommand\Lt{\ensuremath{\textrm{L-tight}}}
\newcommand\hLt{\ensuremath{\textrm{hL-tight}}}
\newcommand\sCt{\ensuremath{\sigma\textrm{-cmpt-tight}}}
\newcommand\wt{\ensuremath{\omega\textrm{-tight}}}
\newcommand\Ctt{\ensuremath{C_2\textrm{-tight}}}
\newcommand\wNt{\ensuremath{\omega\textrm{N-tight}}}
\newcommand\wDt{\ensuremath{\omega\textrm{D-tight}}}

\newcommand\hLC{\ensuremath{\textrm{hL}}}
\newcommand\sCC{\ensuremath{\sigma\textrm{-cmpt}}}
\newcommand\wC{\ensuremath{\omega}}
\newcommand\CtC{\ensuremath{C_2}}
\newcommand\wNC{\ensuremath{\omega\textrm{N}}}
\newcommand\wDC{\ensuremath{\omega\textrm{D}}}

\newtheorem{theorem}{Theorem}[section]
\newtheorem{corollary}[theorem]{Corollary}
\newtheorem{lemma}[theorem]{Lemma}

\newtheorem{question}[theorem]{Question}
\newtheorem{ex}[theorem]{Example}

\def\arhang{Arhangel'skii}
\def\szent{Szentmikl{\'o}ssy}
\def\con{\subseteq}
\def\CL#1{\overline{#1}}
\newcommand\rationals{\mathbb Q}
\newcommand\reals{\mathbb R}
\def\mathcal#1{\mathscr{#1}}

\newif\iflabels

\labelstrue

\begin{document}
\author{Istvan Juh\'asz}

\address{Alfr\'ed R\'enyi Institute of Mathematics, Hungarian Academy of Sciences}
\email{juhasz@renyi.hu}
\author{Jan van Mill}
\def\addressjan{
\address{KdV Institute for Mathematics\\
University of Amsterdam\\
Science Park 105-107\\
P.O. Box 94248\\
1090 GE Amsterdam, The Netherlands}
\email{j.vanMill@uva.nl}
\urladdr{http://staff.fnwi.uva.nl/j.vanmill/}}
\addressjan

\title{Variations on countable tightness}

\date{\today}



\begin{abstract}
We consider 9 natural tightness conditions for topological spaces that are all variations on countable tightnes and investigate the interrelationships between them. Several natural open problems are raised.
\end{abstract}

\maketitle

\section{Introduction}\label{introduction}

A space $X$ has \emph{countable tightness}, or is \emph{countably tight}, if its topology is determined by its countable subsets in the following sense: if $x\in X$ is in the closure of a subset $A$ of $X$ then it is in the closure of some countable subset of $A$. Every sequential space is countably tight and so is every locally countable space.
The statement that countably tight compact Hausdorff spaces are sequential is independent of
the usual axioms of set theory.
We refer to Tkachuk~\cite{Tkachuk:eersteboek} for information about the
importance of the concept of countable tightness in function spaces.

In \arhang\ and Stavrova~\cite{ArhangStavrova93} an interesting variation on the notion of countable tightness was considered. They call a space \emph{$\sigma$-compact tight} if its topology is determined,
in the above sense, by its $\sigma$-compact subspaces.
It was shown in~\cite{ArhangStavrova93} that for compact Hausdorff spaces
$\sigma$-compact tightness is actually equivalent to countable tightness.
Their obvious question whether this in fact holds true for all Tychonoff spaces
turned out to be quite an interesting and difficult problem that remains unsolved.
For some partial results on this problem see Dow and Moore~\cite{DowMoore15}.

Inspired by these results we define and study several tightness conditions here that are in the same spirit. These concepts generalize in an obvious way from the countable to higher cardinals.
But we will not consider them now, we will stick strictly to the countable case.

We shall consider the following properties $\mathcal{P}$ that a subspace of a topological space might have:

\begin{enumerate}
\item[$\omega\ensuremath{\textrm{D}}$] Countable discrete;
\item[$\omega\ensuremath{\textrm{N}}$] Countable and nowhere dense;
\item[$C_2$] Second-countable;
\item[$\omega$] Countable;
\item[hL] Hereditarily Lindel\"of;
\item[$\sigma$-cmpt] $\sigma$-compact;
\item[ccc] The countable chain condition;
\item[L] Lindel\"of;
\item[wL] Weakly Lindel\"of.
\end{enumerate}

\noindent We call a space \emph{$\mathcal{P}$-tight}, if for all $x\in X$ and $A\con X$ such that $x\in \CL{A}$, there exists $B\con A$ such that $x\in\CL{B}$ and $B$ has property $\mathcal{P}$.

It is clear that if property $\mathcal{P}$ implies property $\mathcal{Q}$ then every $\mathcal{P}$-tight space is $\mathcal{Q}$-tight. Assume that the property $\mathcal{P}$ that we are interested in is inherited by dense subspaces. Since every space contains a left-separated dense subspace (Juh\'asz~\cite{juhasz}), it follows that $\mathcal{P}$-tightness and $\mathcal{Q}$-tightness coincide, where $\mathcal{Q}$ is the property of being both $\mathcal{P}$ and left-separated. This observation helps to narrow down the number of properties to consider. For example, missing in our list is the property $\mathcal{S}$
of having countable spread. We left it out because for every space $X$ we have that $X$ is $\mathcal{S}$-tight if and only if $X$ is \hLt. To see this, simply observe that every left-separated space of countable spread is hereditarily Lindel\"of (Juh\'asz~\cite[2.12]{juhasz}).

We also note that, as second countable spaces are separable, the property $\omega C_2$-tight is equivalent to
$C_2$-tight and hence is also left out.

The aim of this note is to investigate the interrelationships of these tightness conditions and to raise some open problems. An analogous study   of variations of $\omega$-boundedness was carried out in Juh\'asz, van Mill and Weiss~\cite{JuhaszMillWeiss11} and Juh\'asz, Soukup and \szent~\cite{JuhaszSoukupSzentmiklossy15}. To our surprise it turned out that the natural concept of Lindel\"of-tightness is mysterious and difficult. It is not known to us, for example, whether every L-tight space is hL-tight, or ccc-tight.

What we do know and do not know about the above tightness conditions is summarized in the following diagram:
$$
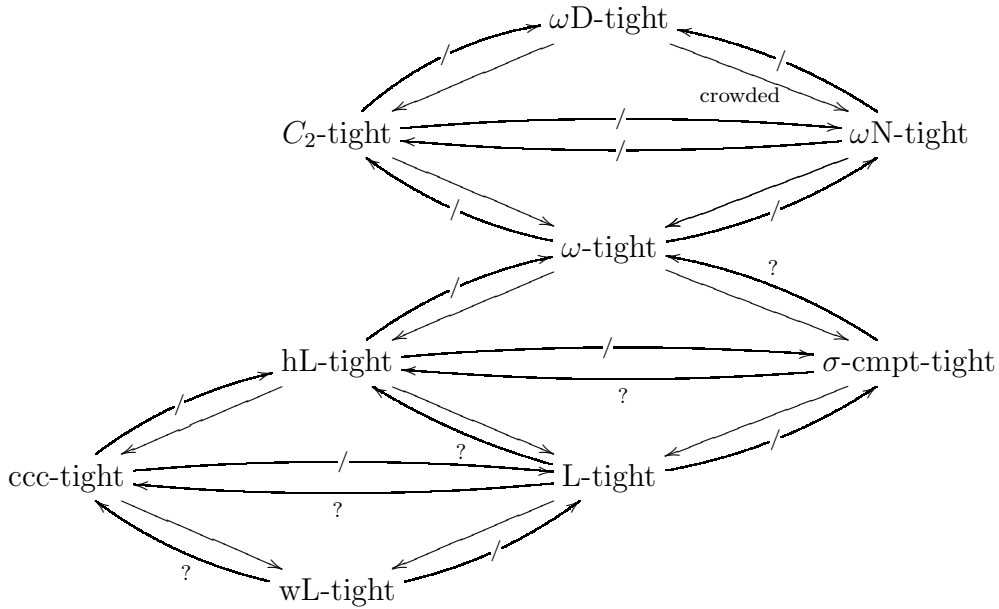
\xymatrix@C=18mm@R=9mm{
 & & \wDt\ar[dr]_{{\mathrm{crowded}}}\ar[dl] &
\\
 &\Ctt\ar[dr]\ar@/^1pc/[ur]|-{/}\ar@/^0.5pc/[rr]|-{/} &  & \wNt\ar[dl]\ar[dl]\ar@/_0.75pc/[ul]|-{/}\ar@/^0.5pc/[ll]|-{/}
\\
 & & \wt\ar[dr]\ar[dl]\ar@/_0.75pc/[ur]|-{{/}} \ar@/^0.75pc/[ul]|-{/} &
\\
& \hLt\ar[dr]\ar[dl]\ar@/^0.5pc/[rr]|-{/}\ar@/^0.75pc/[ur]|-{{/}} & &  \sCt\ar[dl]\ar@/_0.75pc/[ul]_{{?}}\ar@/^0.5pc/[ll]^{{?}}
\\
 \ccct\ar[dr]\ar@/^0.5pc/[rr]|-{/}\ar@/^0.75pc/[ur]|-{{/}}  & & \Lt\ar[dl]\ar@/_0.75pc/[ur]|-{{/}}\ar@/^0.5pc/[ul]^{{{?}}}\ar@/^0.5pc/[ll]^{{?}} &
\\
 &\wLt\ar@/_0.75pc/[ur]|-{{/}}\ar@/^0.75pc/[ul]^{{{?}}} & &
\\
}
$$
\noindent {\bf Acknowledgments}. This note derives from the authors’ collaboration
at the R\'enyi Institute in Budapest in the spring of 2016. The
second-listed author is pleased to thank hereby the Hugarian Academy
of Sciences for its generous support in the framework of the distinguished
visiting scientists program of the Academy and the R\'enyi Institute
for providing excellent conditions and generous hospitality. The
first author also thanks the support of the NKFIH grant no. 113047.

\section{Proofs}

Discrete subsets of crowded spaces are nowhere dense. Hence the implication $\wDt \Longrightarrow \wNt$ indeed holds true for crowded spaces. All other implications in the above diagram need no further explanation. So we concentrate on describing counterexamples that demonstrate that certain implications cannot be reversed. But before we do that, we will point out that one of the problems that we were unable to settle is equivalent to the problem of \arhang\ and Stavrova that we discussed in \S\ref{introduction}.

\begin{lemma}\label{eerstelemma}
For any Hausdorff space $X$, the following statements are equivalent:
\begin{enumerate}
\item $X$ is countably tight.
\item $X$ is both $\sCC$-tight and $\hLC$-tight.
\end{enumerate}
\end{lemma}

\begin{proof}
We only need to prove that (2) implies (1). To this end,
pick an arbitrary $x\in X$, and let $A\con X$ be such that $x\in \CL{A}$.
We may clearly assume without loss of generality that $A$ is left-separated, $\sigma$-compact and \hLC.
Every compact Hausdorff and
left separated space is scattered by Juh\'asz and Gerlits~\cite[Theorem 1]{GerlitsJuhasz78} and every scattered \hLC-space is countable. Hence $A$ is countable, being the countable union of countable sets.
\end{proof}

\begin{corollary}
The following statements are equivalent:
\begin{enumerate}
\item Every $\sCC$-tight Hausdorff  space is countably tight.
\item Every $\sCC$-tight Hausdorff space is \hLC-tight.
\end{enumerate}
\end{corollary}

\noindent Hence the two open problems in the middle part of our diagram are indeed equivalent.

\bigskip

We now turn to describing our counterexamples exemplifying the non-arrows
in our diagram. They are all Tychonoff.
Formally, our tightness conditions are defined for all topological spaces, no separation axioms are implicitly needed in their definitions. But we want our counterexamples to be nice, so from now on we will assume that all topological spaces are Tychonoff.

Remote points will play an important role in their constructions. A \emph{remote point} of a topological space $X$ is a point $p\in \beta X\setminus X$, where $\beta X$ is the \v{C}ech-Stone compactification of $X$, such that for every nowhere dense subset $A$ of $X$, $p\not\in \mathrm{CL}_{\beta X} A$. Van Douwen~\cite{disseen} proved that all non-pseudocompact spaces of countable $\pi$-weight have remote points, a result that was subsequently generalized by Dow~\cite{Dow84} who proved the same result for the class of all non-pseudocompact ccc-spaces of $\pi$-weight at most $\omega_1$. Not all non-pseudocompact spaces have remote points, as was shown by van Douwen and van Mill~\cite{vm:66}.

\newpage

\begin{ex}\label{not_wD}
{ A crowded \CtC-tight space which is  not \wNC-tight, hence not \wDC-tight either.}
\end{ex}
Fix a remote point $p$ of $\rationals$,
the space of rational numbers, and put $X= \rationals \cup \{p\}$
considered as a subspace of $\beta\rationals$. Then $X$ is clearly as required.

\begin{ex}\label{nodec}
{A crowded countable, hence \wC-tight, space which is neither \wNC-tight, nor \CtC-tight.}
\end{ex}
A space $X$ is called \emph{nodec} if all of its nowhere dense subsets are closed.
This clearly implies that all of its nowhere dense subsets are actually closed and discrete.
Van Douwen constructed a countable and crowded nodec space $X$ in~\cite{maximaltopologies}.
But every second-countable subspace of a crowded nodec space is discrete
because it cannot contain a non-trivial convergent sequence, hence $X$ is the example we are after.
In fact, $X$ is neither \wNC-tight, nor \CtC-tight at any of its points.

\par\medskip

To obtain our next example, we introduce a generalization of the well-known Alexandroff duplicate construction.
Let $X$ be any space and fix a pairwise disjoint collection $\mathcal{Y} = \{Y_x : x\in X\}$ of (nonempty) topological spaces such that $X\,\cap\, \bigcup \mathcal{Y} = \emptyset$ and consider the set $Z(X,\mathcal{Y})= X\,\cup\, \bigcup\mathcal{Y}$. If $x\in A \con X$, then let
$$
    W(x,A) = A \cup \bigcup {\big \{}Y_{x'} : x'\in A\setminus \{x\}{\big \}}.
$$
We topologize $Z(X,\mathcal{Y})$ as follows. For every $x\in X$, $Y_x$ is a clopen subsspace of $Z(X,\mathcal{Y})$ whose relative topology coincides with the original topology on $Y_x$. A basic open
$Z(X,\mathcal{Y})$-neighborhood of $x\in X$ has the form $W(x,U)$, where $U$ is any open neighborhood of $x$ in $X$.
It is obvious that if $X$ is crowded then the set $\bigcup_{x\in X} Y_x$ is dense open in $Z(X,\mathcal{Y})$,
hence $X$ is nowhere dense in $Z(X,\mathcal{Y})$.

It is also easy to check that $Z(X,\mathcal{Y})$ is Tychonoff, provided that $X$ and all the $Y_x$  are.
Its topology is inspired by the Alexandroff duplicate of $X$ where each $x\in X$ corresponds to a specific isolated point, its duplicate. That point is simply "blown up" to the space~$Y_x$.

\begin{lemma}\label{derdelemma}
If $X$ is countably tight, moreover $X$ and all the $Y_x$  are crowded, then the space $Z(X,\mathcal{Y})$ is
$\wNC$-tight at every point of $X$.
\end{lemma}

\begin{proof}
Fix a point $p$ in $X$, and let $A$ be any subset of $Z(X,\mathcal{Y})$ such that $p\in \CL{A}$. We may assume without loss of generality that $p\not\in A$. If $p\in \CL{A\cap X}$ then we are done since $X$ is a countably tight and nowhere dense in $Z$. Hence we may assume without loss of generality that $A\con Y = \bigcup_{x\in X} Y_x$.

Consider the projection map $\pi : Y \to X$ defined by $\pi(y) = x$ for $y \in Y_x$. It is clear from the definitions that
$\pi$ is continuous, hence we have $p \in \overline{\pi[A]}$.
Thus there is a countable subset $S \con \pi[A]$ with $p \in \overline{S}$. Clearly, we may assume that $p \notin S$.
For every $x\in S$ we may fix an element $b(x) \in A\cap Y_x$. The set $B = \{b(x) : x \in S\}$ is countable and nowhere dense in $Z(X,\mathcal{Y})$. To see this, simply observe that every $Y_x$ is crowded.

We claim that $p$ is in the closure of $B$. To prove this, consider any basic open neighborhood $W(p,U)$ of $p$ in $Z(X,\mathcal{Y})$, where $U$ is an open neighborhood of $p$ in $X$. But then $S \cap U \ne \emptyset$,
and for every $x \in S \cap U$  we clearly have that $b(x) \in W(p,U) \cap B$.
\end{proof}

\begin{ex}
{An \wNC-tight space which is not \wDC-tight.}
\end{ex}
We again consider the space of rational numbers $\rationals$, and fix a remote point $p$ of $\rationals$. Put $X= \rationals \cup \{p\}$, and in the above construction consider a pairwise disjoint collection of spaces $\mathcal{Y} = \{Y_x : x\in X\}$such that $X\cap \bigcup Y = \emptyset$ and each $Y_x$ is a topological copy of $\rationals$. We claim that the space
$Z(X,\mathcal{Y})$ with the topology that we just discussed is the space we are looking for.
That $Z(X,\mathcal{Y})$ is not $\wDC$-tight is clear since $p$, being remote, is not in the closure of any countable discrete subset of $\rationals$. Observe that $Z(X,\mathcal{Y}) \setminus \{p\}$ is homeomorphic to $\rationals$, being a countable crowded second-countable space. Hence to prove that $Z(X,\mathcal{Y})$ is $\wNC$-tight, we only need to check this at the point $p$. But this is a straight forward consequence of Lemma~\ref{derdelemma}.

Since countable discrete spaces are second countable, the following example is actually a strengthening of
the previous one.

\begin{ex}
{An \wNC-tight space which is not \CtC-tight.}
\end{ex}
Let $X$ be the countable nodec space from Example \ref{nodec}.
To apply the above duplicate construction, consider a pairwise disjoint collection $\mathcal{Y} = \{Y_x : x\in X\}$ of spaces
such that $X\cap \bigcup_{x\in X} Y_x = \emptyset$ and each $Y_x$ is a topological copy of $\rationals$. We claim that
$Z(X,\mathcal{Y})$ is the space which we are looking for. That $Z(X,\mathcal{Y})$ is \wNC-tight is a consequence of Lemma~\ref{derdelemma} and the fact that its subspaces $Y_x$ for $x\in X$ are clopen and \wNC-tight.

Now pick an arbitrary $x\in X$. Then we know form Example \ref{nodec} that $X$ is not \CtC-tight at $x$.
But $X$ is a subspace of $Z(X,\mathcal{Y})$ and so this fact is clearly preserved in $Z(X,\mathcal{Y})$.

\par\medskip

Although the following example does not demonstrate that one of our implications in the diagram cannot be reversed, it solves a natural problem and is therefore included.

\begin{ex}
{A space which is both \CtC-tight and \wNC-tight but not \wDC-tight.}
\end{ex}
Let $A$ be a closed and nowhere dense copy of $\rationals$ in $\rationals$.
(Take, for instance, $\mathbb{Q} \times \{0\}$ in $\mathbb{Q}^2$.)
Observe that the closure of $A$ in $\beta\rationals$ is $\beta A$. Let $p$ be a remote point of $A$ which we think of as a point of $\beta\rationals$, and let $X = \rationals \cup \{p\}$
taken as a subspace of $\beta \mathbb{Q}$.
We claim that $X$ is the space we are looking for. First, it is trivial that $X$ is not \wDC-tight because its subspace $A \cup \{p\}$ is homeomorphic to
the space constructed in Example \ref{not_wD} which is not \wDC-tight. It is also obvious that $X$ is \CtC-tight. To see that that $X$ is also \wNC-tight,
we only have to check this property at $p$.

So, assume that $B$ is a subset of $\rationals$ such that $p\in \mathrm{CL}_{\beta\rationals} B$. If $p$ is in the closure of $A\cap B$, then we are done. If not, then
$p \in \mathrm{CL}_{\beta\rationals} (B \setminus A)$, hence we may assume without loss of generality that $A\cap B = \emptyset$.
We let $\CL{B}$ denote the closure of $B$ in $\rationals$.  Then for every closed neighborhood $U$ of $p$ in $\beta\mathbb{Q}$ we have
$p\in \mathrm{CL}_{\beta\rationals}(U \cap A) \cap  \mathrm{CL}_{\beta\rationals} B$, consequently $U \cap A \cap \CL{B} \ne \emptyset$ because
disjoint closed sets in $\mathbb{Q}$ have disjoint closures in $\beta\mathbb{Q}$. But this means that
$p \in \mathrm{CL}_{\beta\rationals} (A\cap \CL{B})$. Clearly, $B \con \mathbb{Q} \setminus A$ implies that
$A\cap \CL{B}$ is a nowhere dense subset of $\CL{B}$. It is standard to show then that there is a (countable) discrete subset $D$ of $B$ such that $A\cap \CL{B} \con \CL{D}$.
But then $p\in \mathrm{CL}_{\beta\rationals} D$ and $D \con B$ is, of course, countable and nowhere dense in $X$.

\medskip

\begin{ex}
{An \hLC-tight space which is not \sCC-tight.}
\end{ex}
Let $X$ be an $L$-space which is left-separated in type $\omega_1$ and has weight $\omega_1$. Clearly,
the existence of any $L$-space (see e.g. Moore~\cite{Moore06}) implies the existence of such a space $X$.
We can also assume that $X$ is nowhere separable, i.e. countable subsets in $X$ are nowhere dense.
Indeed, just take a maximal pairwise disjoint family consisting of separable open sets and throw away its union.
Of course, then $X$ cannot be compact by not being scattered, using again \cite{GerlitsJuhasz78}, hence $X$ is not pseudocompact, either.
So, $X$ has a remote point $p$ by the above mentioned result of Dow~\cite{Dow84}. Let us now
put $Y= X\cup \{p\}$, considered as the subspace of $\beta X$. Then $Y$ is clearly \hLC-tight and not countably tight. Hence $Y$ is not \sCC-tight either by Lemma~\ref{eerstelemma}.

\medskip

\begin{ex}
{A ccc-tight space which is not Lindel\"of-tight.}
\end{ex}

For any space $X$ we let $F[X]$ denote its Pixley-Roy hyperspace. We recall that a subset $X$ of $\reals$ is called \emph{$\omega_1$-dense} if every nonempty interval in $\reals $ intersects it in a set of size $\omega_1$.
In particular, then $|X| = \omega_1$.

\begin{lemma}\label{tweedelemma}
If $X$ is an $\omega_1$-dense subspace of $\reals$, then $F[X]$ is ccc and every Lindel\"of subspace of $F[X]$ is countable and nowhere dense.
\end{lemma}

\begin{proof}
That $F[X]$ is ccc is well-known and implicit in Pixley and Roy~\cite{PixleyRoy69}.

Suppose $\mathcal{A}\con F[X]$ is uncountable. We claim that $\mathcal{A}$ has an uncountable subset that is
closed and discrete in $F[X]$.
We may assume that, for some natural number $n$, all members of $\mathcal{A}$ have cardinality $n$. Write each $A\in \mathcal{A}$ as $A = \{x_1^A,\dots, x_n^A\}$, where $x_1^A < \cdots < x^A_n$. Then there is $\varepsilon > 0$ such that for an uncountable subcollection $\mathcal{B}$ of $\mathcal{A}$ we have that   $|x_i^A-x_{i+1}^A|\ge \varepsilon$ for all $1\le i \le n{-}1$ whenever $A\in \mathcal{B}$. 

But $\mathcal{B}$ is closed and discrete in $F[X]$. Indeed, consider any $H \in F[X]$.
If $|H| \ge n$ then even $[H,X] \cap \mathcal{A}$ has at most one element. (Recall that
a basic neighborhood of a point $H$ in $F[X]$ has the form
$$
    [H,U]=\{G\in F[X] : H\con G\con U\},
$$
where $U$ is any open neighborhood of $H$ in $X$.)
If, on the other hand, $|H| < n$ then fix for
each $x \in H$ an open ball $U_x$ about $x$ of diameter $< \varepsilon$ and put $U = \bigcup_{x \in H} U_x$.
We claim that then $[H,U] \cap \mathcal{B} = \emptyset$. Indeed,  $A \in \mathcal{B} \cap [H,U]$ would imply
$|A \cap U_x| \ge 2$ for some $x \in H$, hence we would have distinct $i,j \le n$ with $x_i^A, x_j^A \in U_x$.
But this is clearly impossible. 

It obviously follows then that no uncountable subspace of $F[X]$ is Lindel\"of. (We have not used so far the
assumption that $X$ is  $\omega_1$-dense.)

Now let $\mathcal{A}\con F[X]$ be countable and fix any basic open set $[H,U]$.  Since $\bigcup \mathcal{A}$ is countable and 
$U$ is uncountable, we may pick an element $p\in U\setminus \bigcup \mathcal{A}$. Then $[H \cup \{p\},U]$ is a nonempty open subset of $[H,U]$ that misses $\mathcal{A}$, hence $\mathcal{A}$ is nowhere dense in $F[X]$.

\end{proof}

Obviously,  the weight of $F[X]$ is $\omega_1$. Also, it is  not pseudocompact. Indeed, it is of first category being
crowded and  $\sigma$-discrete, while pseudocompact spaces are Baire, see e.g.  Engelking \cite{E}, 3.10.23 (ii). But then
$F[X]$ has a remote point, say $p$. 
We claim that $Z = F[X] \cup \{p\}$, as a subspace of $\beta F[X]$, is the space we are looking for. That $Z$ is not \Lt\ is an immediate consequence of Lemma~\ref{tweedelemma}. 

To prove that $Z$ is \ccct, we first remark that because $F[X]$ is first countable, this need only be checked at $p$. So assume that $A \con F[X]$ has $p$ in its closure. Consider the closure $B$ of $A$ in $F[X]$ and let $C$ denote its interior. Then $D=B\setminus C$ is nowhere dense in $F[X]$, hence $p$ is not in the closure of $D$.
Consequently, $p$ is in the closure of $A\cap C$ which is ccc, being dense in the open set $C \con F[X]$. Hence we are done.

\subsection*{Discussion}
As we mentioned in \S\ref{introduction}, we were unable to solve the following two basic problems on Lindel\"of-tightness that are mentioned in the diagram:

\begin{question}\label{OpenProblems}${}$
\begin{enumerate}
\item[(A)] Is every \Lt\ space \hLt?
\item[(B)] Is every \Lt\ space \ccct?
\end{enumerate}
\end{question}

Let us repeat what we said earlier: if the answer to Question~\ref{OpenProblems}(A) is in the affirmative, then so is the answer to the \arhang-Stavrova question whether $\sigma$-compact tightness implies countable tightness.

The following problem is interesting in its own right and could be easier to tackle than Question~\ref{OpenProblems}.

\begin{question}
Is every space that is both \Lt\ and \ccct, \hLt?
\end{question}


\def\cprime{$'$}
\makeatletter \renewcommand{\@biblabel}[1]{\hfill[#1]}\makeatother

\end{document}